\documentclass[12pt]{amsart}
\usepackage{amsmath,amssymb,amsbsy,amsfonts,amsthm,latexsym,mathabx,
            amsopn,amstext,amsxtra,euscript,amscd,stmaryrd,mathrsfs,
            cite,array,mathtools,enumerate}

\usepackage[colorlinks,linkcolor=blue,anchorcolor=blue,citecolor=blue,backref=page]{hyperref}
\usepackage{scalerel}
\newcommand\shortslash{\stretchrel*{$/$}{\textsc{e}}}

\usepackage[norefs,nocites]{refcheck}       
\usepackage{url}
\usepackage{color}
\usepackage{comment} 
\usepackage[english]{babel}

\usepackage{mathtools}
\usepackage{todonotes}
\begin{document}

\newtheorem{theorem}{Theorem}
\newtheorem{lemma}[theorem]{Lemma}
\newtheorem{claim}[theorem]{Claim}
\newtheorem{cor}[theorem]{Corollary}
\newtheorem{prop}[theorem]{Proposition}
\newtheorem{definition}[theorem]{Definition}
\newtheorem{question}[theorem]{Question}
\newtheorem{remark}[theorem]{Remark}
\newcommand{\hh}{{{\mathrm h}}}

\numberwithin{equation}{section}
\numberwithin{theorem}{section}

\def\sssum{\mathop{\sum\!\sum\!\sum}}
\def\ssum{\mathop{\sum\ldots \sum}}

\def \balpha{\boldsymbol\alpha}
\def \bbeta{\boldsymbol\beta}
\def \bgamma{{\boldsymbol\gamma}}
\def \bomega{\boldsymbol\omega}

\def\sssum{\mathop{\sum\!\sum\!\sum}}
\def\ssum{\mathop{\sum\ldots \sum}}
\def\dsum{\mathop{\sum\  \sum}}
\def\iint{\mathop{\int\ldots \int}}

\def\squareforqed{\hbox{\rlap{$\sqcap$}$\sqcup$}}
\def\qed{\ifmmode\squareforqed\else{\unskip\nobreak\hfil
\penalty50\hskip1em\null\nobreak\hfil\squareforqed
\parfillskip=0pt\finalhyphendemerits=0\endgraf}\fi}

\newfont{\teneufm}{eufm10}
\newfont{\seveneufm}{eufm7}
\newfont{\fiveeufm}{eufm5}
%
%
\newfam\eufmfam
     \textfont\eufmfam=\teneufm
\scriptfont\eufmfam=\seveneufm
     \scriptscriptfont\eufmfam=\fiveeufm
%
%
\def\frak#1{{\fam\eufmfam\relax#1}}

\def\fK{\mathfrak K}
\def\fT{\mathfrak{T}}

\def\ges{\gtrsim}
\def\les{\lesssim}
\def\fA{{\mathfrak A}}
\def\fB{{\mathfrak B}}
\def\fC{{\mathfrak C}}
\def\fD{{\mathfrak D}}

\newcommand{\sX}{\ensuremath{\mathscr{X}}}

\def\vec#1{\mathbf{#1}}
\def\dist{\mathrm{dist}}
\def\vol#1{\mathrm{vol}\,{#1}}

\def\squareforqed{\hbox{\rlap{$\sqcap$}$\sqcup$}}
\def\qed{\ifmmode\squareforqed\else{\unskip\nobreak\hfil
\penalty50\hskip1em\null\nobreak\hfil\squareforqed
\parfillskip=0pt\finalhyphendemerits=0\endgraf}\fi}

\def\sA{\mathscr A}
\def\sB{\mathscr B}
\def\sC{\mathscr C}
\def\sD{\Delta}
\def\sE{\mathscr E}
\def\sF{\mathscr F}
\def\sG{\mathscr G}
\def\sH{\mathscr H}
\def\sI{\mathscr I}
\def\sJ{\mathscr J}
\def\sK{\mathscr K}
\def\sL{\mathscr L}
\def\sM{\mathscr M}
\def\sN{\mathscr N}
\def\sO{\mathscr O}
\def\sP{\mathscr P}
\def\sQ{\mathscr Q}
\def\sR{\mathscr R}
\def\sS{\mathscr S}
\def\sU{\mathscr U}
\def\sT{\mathscr T}
\def\sV{\mathscr V}
\def\sW{\mathscr W}
\def\sX{\mathscr X}
\def\sY{\mathscr Y}
\def\sZ{\mathscr Z}

\def\cA{{\mathcal A}}
\def\cB{{\mathcal B}}
\def\cC{{\mathcal C}}
\def\cD{{\mathcal D}}
\def\cE{{\mathcal E}}
\def\cF{{\mathcal F}}
\def\cG{{\mathcal G}}
\def\cH{{\mathcal H}}
\def\cI{{\mathcal I}}
\def\cJ{{\mathcal J}}
\def\cK{{\mathcal K}}
\def\cL{{\mathcal L}}
\def\cM{{\mathcal M}}
\def\cN{{\mathcal N}}
\def\cO{{\mathcal O}}
\def\cP{{\mathcal P}}
\def\cQ{{\mathcal Q}}
\def\cR{{\mathcal R}}
\def\cS{{\mathcal S}}
\def\cT{{\mathcal T}}
\def\cU{{\mathcal U}}
\def\cV{{\mathcal V}}
\def\cW{{\mathcal W}}
\def\cX{{\mathcal X}}
\def\cY{{\mathcal Y}}
\def\cZ{{\mathcal Z}}
\newcommand{\rmod}[1]{\: \mbox{mod} \: #1}

\def\vr{\mathbf r}

\def\e{{\mathbf{\,e}}}
\def\ep{{\mathbf{\,e}}_p}
\def\em{{\mathbf{\,e}}_m}
\def\en{{\mathbf{\,e}}_n}

\def\Tr{{\mathrm{Tr}}}
\def\Nm{{\mathrm{Nm}}}

 \def\SS{{\mathbf{S}}}

\def\lcm{{\mathrm{lcm}}}

\def\({\left(}
\def\){\right)}
\def\fl#1{\left\lfloor#1\right\rfloor}
\def\rf#1{\left\lceil#1\right\rceil}

\def\mand{\qquad \mbox{and} \qquad}
\hypersetup{breaklinks=true}

\newcommand{\commGi}[2][]{\todo[#1,color=yellow]{Gi: #2}}

\newcommand{\commIg}[2][]{\todo[#1,color=magenta]{Ig: #2}}

\newcommand{\commIl}[2][]{\todo[#1,color=blue]{Il: #2}}

\newcommand{\commSi}[2][]{\todo[#1,color=green]{Si: #2}}




\hyphenation{re-pub-lished}

\parskip 4pt plus 2pt minus 2pt



\def\bfdefault{b}
\overfullrule=5pt

\def \F{{\mathbb F}}
\def \K{{\mathbb K}}
\def \Z{{\mathbb Z}}
\def \Q{{\mathbb Q}}
\def \R{{\mathbb R}}
\def \C{{\\mathbb C}}
\def\Fp{\F_p}
\def \fp{\Fp^*}

\title[Trilinear Exponential Sums]{Bounds of Trilinear and Trinomial Exponential Sums}

  \author[S.  Macourt] {Simon Macourt}
\address{Department of Pure Mathematics, University of New South Wales,
Sydney, NSW 2052, Australia}
\email{s.macourt@unsw.edu.au}

 \author[G. Petridis] {Giorgis Petridis}
\address{Department of Mathematics, University of Georgia, 
Athens, GA 30602, USA}
\email{giorgis.petridis@gmail.com}

 \author[I. D. Shkredov]{Ilya D. Shkredov}
\address{Steklov Mathematical Institute of Russian Academy
of Sciences, ul. Gubkina 8, Moscow, Russia, 119991, Institute for Information Transmission Problems  of Russian Academy
of Sciences, Bolshoy Ka\-ret\-ny Per. 19, Moscow, Russia, 127994, and MIPT, 
Institutskii per. 9, Dolgoprudnii, Russia, 14170}
\email{ilya.shkredov@gmail.com}

\author[I. E. Shparlinski]{Igor E. Shparlinski} 
\address{School of Mathematics and Statistics, University of New South Wales, 
Sydney, NSW 2052, Australia}
\email{igor.shparlinski@unsw.edu.au}

\begin{abstract} 
We prove, for a sufficiently small, subset $\cA$ of a prime residue field an estimate on the number of solutions to the equation $(a_1-a_2)(a_3-a_4) = (a_5-a_6)(a_7-a_8)$ with all variables in $\cA$. We then derive new bounds on trilinear exponential sums and on the total number of residues equaling the product of two differences of elements of $\cA$. We also prove a refined estimate on the number of collinear triples in a Cartesian product of multiplicative subgroups and derive stronger bounds for trilinear sums with all variables in multiplicative subgroups.
\end{abstract}


\maketitle

\section{Introduction}

Let $p$ be a prime and let $\F_p$ be the finite field of $p$ elements. 
Now given three sets $\cX, \cY, \cZ \subseteq \F_p$, and three sequences of  complex weights
 $\alpha= (\alpha_{x})_{x\in \cX}$, $\beta = \( \beta_{y}\)_{y \in \cY}$ and $\gamma =    \(\gamma_{z}\)_{z \in \cZ}$
 supported on
$\cX$,  $\cY$ and 
$\cZ$, respectively, we consider exponential sums
\begin{equation}
\label{eq:SXYZ}
S(\cX, \cY, \cZ;\alpha, \beta, \gamma) = \sum_{x \in\cX} \sum_{y \in \cY}
 \sum_{z\in \cZ}\alpha_{x} \beta_{y} \gamma_{z}\ep(xyz) \,,  
\end{equation}
where $\ep(z) = \exp(2 \pi i z/p)$, the  sets $\cX, \cY, \cZ \subseteq \F_p$ are of cardinalities 
\begin{equation}
\label{eq:sets}
|\cX| = X, \qquad  |\cY| = Y, \qquad |\cZ| = Z \,,
\end{equation}
and  weights satisfy 
\begin{equation}
\label{eq:weights1}
\max_{x \in \cX} |\alpha_{x}| \le 1, \qquad 
\max_{y \in   \cY}  |\beta_{y}| \le 1, \qquad 
\max_{z \in   \cZ}   |\gamma_{z}| \le 1 \,.
\end{equation}

We also define more general sums
\begin{align}\label{eq:TXYZ}
T(\cX, \cY, \cZ ;  \rho, \sigma, \tau) =
\sum_{x \in\cX} \sum_{y \in \cY}
 \sum_{z\in \cZ} \rho_{x,y} \sigma_{x,z} \tau_{y,z}\ep(xyz)
 \end{align}
with some  weights 
$\rho= (\rho_{x,y})$,  $\sigma = (\sigma_{x,z})$ and $\tau=(\tau_{y,z})$
satisfying 
\begin{equation}
\label{eq:weights2}
\max_{(x,y) \in \cX \times \cY} |\rho_{x,y}| \le 1, \quad 
\max_{(x,z) \in \cX \times \cZ} |\sigma_{x,z}| \le 1, \quad 
\max_{(y,z) \in \cY \times \cZ} |\tau_{y,z}| \le 1 \,.  
\end{equation}

Investigation of the sums $S(\cX, \cY, \cZ;\alpha, \beta, \gamma)$ has been initiated by 
Bourgain and Garaev~\cite{BouGar} who have linked them to some problems of 
additive combinatorics and  also have given nontrivial  explicit  bounds on 
these sums which go beyond the classical bilinear bound. 
Several more bounds on these sums, improving and complementing those  
of~\cite{BouGar} can be found in~\cite{Gar2,PetShp,Shkr2}. 

The more general sums $T(\cX, \cY, \cZ ;  \rho, \sigma, \tau)$ have been introduced and estimated 
in~\cite{PetShp}.  Several improvements of the results of~\cite{PetShp} can be found in~\cite{Mac1,Shkr2}. 
Furthermore,  in the special case when the sets $\cX$,  $\cY$ and  $\cZ$ are  multiplicative subgroups on $\F_p^*$, 
bounds of these sums have found applications to new estimates of exponential sums with 
sparse polynomials~\cite{Mac2,MSS}. 

We note that generally speaking the bound on  bilinear sums does not apply to these sums. 
Instead, in~\cite{PetShp} it was shown that they are related to the following combinatorial quantity.  
For any $\cA \subseteq \F_p$, of cardinality $A$, we define
\[
D^\times(\cA) = |\{ (a_1-a_2)(a_3-a_4) = (a_5-a_6)(a_7-a_8):~ a_1\, \dots, a_8 \in \cA \}| \,. 
\]

It follows from~\cite{AYMRS} that $D^\times(\cA) = O(A^{13/2})$  for $A \leq p^{2/3}$, see~\cite[Corollary 2.9]{PetShp}. 
It was shown in~\cite[Theorem~1.3]{PetShp} that under the conditions~\eqref{eq:sets} 
and~\eqref{eq:weights2} we have
\begin{equation}
\label{eq:P-S}
T(\cX, \cY, \cZ ;  \rho, \sigma, \tau)  = O\( p^{1/8} X^{7/8} (YZ)^{29/32}\) \,, 
\end{equation}
provided that $p^{2/3}\ge X \ge Y \ge Z$. 
In some ranges this bound has been improved in~\cite[Theorem~1.2]{Mac1} as
\[
T(\cX, \cY, \cZ ;  \rho, \sigma, \tau)  = O\( p^{3/16} X^{13/16} (YZ)^{7/8}\) \, , 
\]
provided that $X \ge  Y \ge Z$.

If  $X \ge  Y \ge Z$ and also $Y \le p^{48/97}$ then~\eqref{eq:P-S}
has been further improved  in~\cite[Corollary~45]{Shkr2} as 
\begin{equation}
\label{eq:Shkr}
T(\cX, \cY, \cZ ;  \rho, \sigma, \tau)  = O\( p^{1/8} X^{7/8} (YZ)^{29/32-1/3072}\) \, .
\end{equation}
This progress was based on the stronger bound $D^\times(\cA) = O(A^{13/2 - 1/192})$  for $A \leq p^{48/97}$, 
see~\cite[Theorem~41]{Shkr2}.

Multilinear generalisations of the exponential sums in~\eqref{eq:SXYZ} and~\eqref{eq:TXYZ} 
have also been studied, see~\cite{Bou1, Gar2, KM,Mac1,PetShp,Shkr2} and references therein.
Although all the above  works  rely on ideas and results from additive combinatorics and thus work the 
best in prime fields $\F_p$, extension to arbitrary finite fields  can be found in~\cite{Bou2,BouGlib,Moh,Ost}.

Here, as in all previous works, we improve the state-of-the-art on $D^\times(\cA)$ for sufficiently small $A$, see Theorem~\ref{thm:D_improved_new} below, 
and then use some previously known estimates and obtain new concrete bounds on the exponential sums in~\eqref{eq:SXYZ} and~\eqref{eq:TXYZ}, 
see  Theorem~\ref{thm:STsums} below. This progress is rooted at the eigenvalue method of Shkredov.

For the special case of multiplicative subgroups, we achieve further progress by an altogether different method: by improving existing bounds on the number of collinear triples of Cartesian products of multiplicative subgroups, see Theorems~\ref{thm:subgroups} 
and~\ref{thm:Tgroup} below. Our results lead to small improvements on bounds on trinomial exponential sums, see~\cite{MSS} and Corollary~\ref{cor:Bound3}.

We also bound from below the cardinality of the set
\[
(\cA-\cA) (\cA-\cA) = \{ (a_1-a_2)(a_3-a_4) :~ a_1\, \dots, a_4 \in \cA \} 
\,,
\]
see Theorem~\ref{thm:ProdDiff} below, improving~\cite[Theorem~27]{MPR-NRS}.

\section{Notation}

In what follows, it  is convenient to introduce notation $A\les B$ and $B\ges A$ as equivalents
of $A \le p^{o(1)} B$ as $p\to \infty$; and $A \sim B$ as equivalent to $A \les B$ and $B \les A$. 

We also recall that  the notations $A=O(B)$, $A\ll B$ and $B \gg A$ are each equivalent to the
statement that the inequality $A\le c\,B$ holds with a
constant $c>0$ which is absolute throughout this paper.

\section{Some combinatorial quantities}

Given two sets $\cU, \cV \subseteq \F_p$ we define

\begin{itemize}
\item $E^{\ast} (\cU, \cV)$ where $\ast  \in \{+, -, \times, \shortslash \}$ is one of the four arithmetic operations  
as the number of solutions to the equation  
\[
u_1\ast v_1 =u_2 \ast v_2  , \quad  \quad  u_1, u_2   \in \cU, \  v_1, v_2   \in \cV \,; 
 \]
\item $E^{\ast}_3 (\cU, \cV)$ where  $\ast  \in \{+, -, \times, \shortslash \}$ as the number of solutions to the equation  
\[
u_1\ast v_1 =u_2 \ast v_2 =u_3 \ast v_3  , \quad  \quad  u_1, u_2,u_3   \in \cU, \  v_1, v_2, v_3   \in \cV \,; 
 \]
\item  $D^\times (\cU, \cV)$ as the number of solutions to the equation  
\begin{align*}
(u_1-v_1)& (u_2 - v_2)   =  (u_3-v_3)(u_4-v_4), \\  u_i & \in \cU, \  v_i  \in \cV, \quad i=1,2,3,4 \,;
\end{align*}
 \item  $\widetilde D^\times (\cU, \cV)$ as the number of solutions to the equation  
\begin{align*}
(u_1-u_2)(v_1 - v_2) &=  (u_3-u_4)(v_3-v_4) \ne 0, \\ \ u_i \in \cU,  \  v_i & \in \cV,\quad   i=1,2,3,4 \,; 
\end{align*}
\item $\widetilde T(\cU,\cV)$ as the number of collinear triples in $\cU \times \cV$ with slopes in $\F_p^*$, that is,  the number of solutions to the equation
\begin{align*}
(u_1-u_2)& (v_1-v_2)=(u_1-u_3)(v_1-v_3)\neq 0, \\ u_i & \in \cU, \ v_i \in \cV, \ i=1,2,3 \,;
\end{align*}  
\item  $N(\cU, \cV,\cW)$ as the number of solutions to 
\[
u_1(v_1-w_1) = u_2 (v_2-w_2), \qquad u_1, u_2 \in \cU, \ v_1,v_2 \in \cV, \ w_1,w_2 \in \cW \, . 
\]  
\end{itemize}

In the case of equal sets $\cU =\cV$, we write
\[
  E^{\ast} (\cU, \cU) = E^{\ast} (\cU)  \text{ and }  E^{\ast}_3 (\cU, \cU) = E^{\ast}_3 (\cU) \,.
\] 

Clearly 
\[
 \widetilde D^\times (\cU) \le D^\times (\cU) \le  \widetilde  D^\times (\cU)  + 4 |\cU|^6 \,. 
 \]

Note that expressing $ N(\cU, \cV, \cW) $ and $D^\times (\cU, \cV)$ via multiplicative character sums, by the Cauchy inequality we immediately derive
\begin{equation}
\label{eq:N ED}
 N(\cU, \cV, \cW)  \le \sqrt{E^{\times}\(\cU\) D^\times(\cV, \cW) } \,.
 \end{equation}

We also define these quantities with functions instead of sets. For example, for a function $F: \F_p \to \R$ and   $\ast  \in \{+, -, \times, \shortslash \}$  we define
\[
E^\ast_3(F) = \sum_{\substack{x_1, y_1, x_2 ,  y_2 , x_3 , y_3 \in \F_p \\ x_1\ast y_1 = x_2 \ast y_2  = x_3 \ast y_3}}
F(x_1) F(y_1) F(x_2) F(y_2) F(x_3) F(y_3)\,, 
\]
and also given a set $\cU \subseteq \F_p$ we define
\[
E^\ast(F, \cU) = \sum_{\substack{x_1, x_2 \in \F_p, \ u_1.u_2 \in\cU \\ x_1\ast u_1 = x_2 \ast u_2}} F(x_1)  F(x_2) \,.
\]

\section{A new bound on $D^\times(\cA)$}
\label{sec:new D}

We use results based on 
the eigen\-value me\-thod of Shkredov~\cite{Shkr2}
 to prove a new upper bound on $D^\times(\cA)$. Throughout the proof we use the convention that sets are written in curly script capital letters and their cardinality in roman script capital letters, 
for example, $|\cA| = A$. 

We start with a result which can be of independent interest. 

\begin{lemma}
\label{lem:gen_eigenvalues} 	
	Let $F: \F_p \to \R$ be a non--negative function and let $K\ge 1$  be an arbitrary parameter. 
	Suppose that for   $\ast  \in \{+, -, \times, \shortslash \}$  and any set $\cS \subseteq \F_p$	one has 
\begin{equation}\label{cond:energy_gen_eigenvalues}	
	E^\ast (F,\cS) \le K S^{3/2} \,.
\end{equation}
	Then  
\[
	E^\ast(F) \les \(E^\ast _3\)^{6/13} (F) K^{2/13} \| F\|_1^{12/13} \,,
\]
where 
\[ \| F\|_1 = \sum_{x \in \F_p} F(x)\,.
\]
\end{lemma}  

\begin{proof} By the Dirichlet principle, there exists 
some $\Delta> 0 $  such that 
\[
E^\ast(F) \sim  \Delta^4 E^\ast(\cB)\,,
\] 
where $\cB = \{ x:~\Delta < F(x) \le 2\Delta \}$ is the dyadic level set of the function $F$. 
	Then 
\[
B \Delta \le \|F\|_1 \mand E^\ast_3 (\cB) \le \Delta^{-6} E^\ast_3 (F)
\] and for any $\cS \subseteq \F_p$, by our assumption one has 
	\[
		\Delta^2 E^\ast(\cB,\cS) \le E^\ast(F,\cS) \le K S^{3/2} \,.
	\]
	Applying~\cite[Theorem~34]{MPR-NRS} (which holds for any type of energy) with 
	\[D_1 = B^{-3} \Delta^{-6} E^\ast_3 (F) \mand D_2 = \Delta^{-2} B^{-1} K\,,
	\] 
	we obtain 
\begin{align*}
	\(E^\ast(\cB)\)^{13} & \les D^{6}_1 D^{2}_2 B^{32} \ll \(B^{-3} \Delta^{-6} E^\ast_3 (F)\)^6 (\Delta^{-2} B^{-1} K)^2 B^{32} \\
	& = \(E^\ast_3 (F)\)^6 K^2 B^{12} \Delta^{-40} \le \(E^\ast _3 (F)\)^6 K^2 \| F\|_1^{12} \Delta^{-52} \,.
\end{align*}
In other words,
\[
	\(E^\ast (F)\)^{13} \sim\(E^\ast (\cB)\)^{13} \Delta^{52} \les \(E^\ast_3 (F)\)^6 K^2 \| F\|_1^{12} 
\]
	as required. 
\end{proof}

In the proof of Theorem~\ref{thm:D_improved_new}  below we apply Lemma~\ref{lem:gen_eigenvalues}  to the function 
\[
 r_{\cA-\cA} (x) = |\{ (a,b) \in \cA \times \cA :~ a-b =x \}|
\]
and to multiplicative energy.
\begin{lemma}
		\label{lem:next}
	Let  $\cA \subseteq \F_p$  be of cardinality $A \le p^{2/3}$. Then 
\[
	E^{\shortslash}_3 \(r_{\cA-\cA}\) \ll  A^9 \log A \,,
\]  
	and for any set $\cS \subseteq \F_p$ with $S \le A^2$ one has 
\[
	E^{\shortslash} \(r_{\cA-\cA}, \cS\) \ll \frac{A^4 S^2}{p} + A^3 S^{3/2}  \,.
\]
\end{lemma}

\begin{proof}
	For $a,b \in \F_p$ set
	\[
	  r_{(\cA-a)/(\cA-b)} (x) = |\{ (c,d) \in \cA \times \cA : (c-a) / (d-b) = x \}|\,.
	\] 
From the definition of $E^{\shortslash}_3 (r_{\cA-\cA})$ collecting terms with the same value of 
$x_1 / y_1 = x_2  / y_2 = x_3 / y_3 = z$ we have
\begin{align*}
	E^{\shortslash}_3 \(r_{\cA-\cA}\)  &=  \sum_{z\in \F_p}   \sum_{x_1 x_2 ,   x_3  \in \F_p }
r_{\cA-\cA}(x_1) r_{\cA-\cA}(z x_1) \ r_{\cA-\cA} (x_2)    r_{\cA-\cA}(z y_2) \\
& \qquad \qquad \qquad \qquad \qquad \qquad  \qquad \qquad \quad  r_{\cA-\cA}(x_3) r_{\cA-\cA}(z x_3)\\
&  =  \sum_{z\in \F_p}  \( \sum_{x  \in \F_p }
r_{\cA-\cA}(x) r_{\cA-\cA}(xz)\)^3
\end{align*}
We now observe that 
\begin{align*}
 \sum_{x  \in \F_p }
r_{\cA-\cA}(x) r_{\cA-\cA}(z x)  & =    |\{ (a,b,c,d) \in \cA \times \cA : (c-a) / (d-b) = z \}| \\ & = \sum_{a,b \in \cA} r_{(\cA-a)/(\cA-b)} (z) \,.
\end{align*}

	Using the H\"older inequality 
\begin{equation}\label{eq:Holder}	
\begin{split}
	E^{\shortslash}_3 \(r_{\cA-\cA}\) &= \sum_{z\in \F_p} \left( \sum_{a,b\in \cA} r_{(\cA-a)/(\cA-b)} (z) \right)^3\\
	& \le A^4  \sum_{a,b\in \F_p} \sum_{z\in \F_p}   r^3_{(\cA-a)/(\cA-b)} (z)  \,.
	\end{split} 
\end{equation} 

Clearly 
\begin{align*}
& r^3_{(\cA-a)/(\cA-b)} (z)\\
& \qquad =
 \left| \left \{ c_1, d_1, c_2, d_2, c_3, d_3 ) \in \cA^6 :~\frac{c_1-a }{d_1-b} 
=\frac{c_2-a }{d_2-b} =\frac{c_3-a }{d_3-b}  = z \right\} \right|. 
\end{align*}
Thus 
\begin{align*}
   \sum_{z\in \F_p}  &r^3_{(\cA-a)/(\cA-b)} (z)\\
& =
 \left| \left \{ c_1, d_1, c_2, d_2, c_3, d_3 ) \in \cA^6 :~\frac{c_1-a }{d_1-b} 
=\frac{c_2-a }{d_2-b} =\frac{c_3-a }{d_3-b}    \right\} \right|.
\end{align*}
Therefore we see that 
\begin{equation}\label{eq:sumr3}	
 \sum_{a,b\in \F_p} \sum_{z\in \F_p} r^3_{(\cA-a)/(\cA-b)} (z) = Q(\cA),
\end{equation} 
where 
\[
Q(\cA) =   \left| \left \{ (a_1, \ldots, a_8 ) \in \cA^8 :~\frac{a_1-a_2}{b_1-b_2} 
= \frac{a_1-a_3}{b_1-b_3} = \frac{a_1-a_4}{b_1-b_4} \right\} \right|
\]
is the number of ordered collinear quadruples
\[
\((a_1, b_1), (a_2, b_2), (a_3, b_3), (a_4, b_4)\)  \in \(\cA \times \cA\)^4. 
\]
By~\cite[Theorem~11(2)]{MPR-NRS} we have 
\[
Q(\cA)  \ll \frac{A^8}{p^2} + A^5 \log A \ll  	A^9 \log A
\]
which together with~\eqref{eq:Holder}	 and~\eqref{eq:sumr3}	 gives the first inequality. 

To 	obtain the second inequality just apply the modern form of the incidence result
of Rudnev~\cite{Rudnev}, 
see, for example,~\cite[Theorem~10]{Shkr2}.
This completes the proof. 
\end{proof}

\begin{theorem}
\label{thm:D_improved_new}
	Let $\cA\subseteq \F_p$ be of cardinality $A \le p^{1/2}$.
Then 
	\[
	D^\times (\cA) \les  A^{84/13} \,. 
	\] 
\end{theorem}

\begin{proof}
We begin by noting that
\[
D^\times(\cA) = E^\times(r_{\cA-\cA}) \,.
\]
By the first inequality of Lemma~\ref{lem:next} we have 
\[
E^\times_3(r_{\cA-\cA}) \les A^9
\]
while  for all $\cS$ satisfying $S \leq A^2 \leq p$ the second inequality  becomes
\[
E^\times (r_{\cA-\cA}, \cS) \ll \frac{A^4 S^2}{p} + A^3 S^{3/2} \ll A^3 S^{3/2} \,.
\]

By the remark after~\cite[Theorem~34]{MPR-NRS}, we only have to confirm~\eqref{cond:energy_gen_eigenvalues} for sets $\cS$ of cardinality $S \ll A^4 / E^\times(A) \leq A^2$, and so Lemma~\ref{lem:next} gives $K=A^3$. Moreover 
\[
\| r_{\cA - \cA} \|_1 = \sum_{x \in \F_p}  r_{\cA - \cA}  (x)  = A^8\,.
\] 
 Substituting all this in Lemma~\ref{lem:gen_eigenvalues} with $F(x) = r_{\cA - \cA}  (x)$ gives
\[
D^\times(\cA) = E^\times(r_{\cA-\cA}) \les A^{84/13} \,, 
\]
which concludes the proof.
\end{proof}

\section{A refined bound on $D^\times(\cG, \cH )$ over subgroups}
\label{sec:Dsubgroups}

Here we  use $\widetilde T(\cG,\cH)$ to give stronger bounds on $D^\times(\cG, \cH )$, where $\cG$ and $\cH$ are multiplicative subgroups.

First we recall the following result~\cite[Theorem~2]{Mit}. 
\begin{lemma} \label{lem:Mit}
Let $\cG$ and $\cH$ be  subgroups of $\F_p^*$ and let 
 $\cM_\cG$ and $\cM_\cH $   
be two complete sets  of distinct coset representatives 
of $\cG$ and $\cH$  in $\F_p^*$.
For an arbitrary set $\varTheta \subseteq\cM_\cG \times \cM_\cH$ such that 
\[
|\varTheta| \le \min\left\{ |\cG||\cH|, \frac{p^3}{|\cG|^2|\cH|^2}\right\}
\]
we have
\[
\sum_{(u,v) \in \varTheta}\left |\{(x,y) \in \cG\times \cH~:~ux+vy=1\}\right| \ll (|\cG||\cH||\varTheta|^2)^{1/3}.
\]
\end{lemma}

\begin{theorem}
	\label{thm:subgroups}
Let $\cG, \cH \subseteq \F^*_p$ be subgroups with $|\cG|=G, |\cH|=H$ and $G \geq H$. Then 
\[
\widetilde T(\cG,\cH) - \frac{G^3 H^3}{p} \ll \left\{ 
\begin{array}{ll}
p^{1/2} G^{3/2} H^2 & \text{if $GH \ge p^{4/3}$,} \\  
\displaystyle{\frac{G^{5/2}H^{5/2}}{p^{1/2}} + H^2 G^2 \log{G} } & \text{if $ p < GH < p^{4/3} $,}  
\\  
G^3H\log{G} & \text{if $GH \le p$.} \end{array}\right. 
\]
\end{theorem}

\begin{proof}
The result for $GH\le p$ is clear from~\cite[Lemma~2.6]{MSS} once we eliminate the contribution from the zero solutions (see also~\cite{Shkr1}). 
Similarly for $GH \ge p^{4/3}$ from~\cite[Theorem~1.1]{Mac1}. We now prove for $p<GH<p^{4/3}$ by following the argument of~\cite{MSS}. 

We can think of $\cM_{\cG}=\F^*_p/\cG$ and similarly for $\cM_\cH$. We now set 
\[
3 \le \Delta = c \, \frac{G^{3/2} H^{3/2}}{p^{3/2}}
\]
for a sufficiently small $c$ to be chosen later (it is better to think of $\Delta$ as a parameter to be chosen later), 
by~\cite[Corollaries 2.3 and 2.5]{MSS} and by observing that the contribution from lines with $ab=0$ is at most $O(G^2H^2)$ we have
\begin{equation}
\begin{split}
 \label{eq:TGH}
\widetilde T(\cG,\cH)  -\frac{G^3H^3}{p} & \ll G^2H^2 + \Delta pGH\\
&\quad  +\sum_{\substack{ a,b \in \F^*_p \\ \iota_{\cG,\cH}(\ell_{a,b})>\Delta}} \iota_{\cG,\cH}(\ell_{a,b})\left(\iota_{\cG,\cH}(\ell_{a,b})- \frac{GH}{p} \right)^2.
	\end{split}
\end{equation}
Let
\[
W=\sum_{\substack{ a,b \in \F^*_p \\ \iota_{\cG,\cH}(\ell_{a,b})>\Delta}} \iota_{\cG,\cH}(\ell_{a,b})\left(\iota_{\cG,\cH}(\ell_{a,b})- \frac{GH}{p} \right)^2.
\]
We now let $\tau > \Delta$ be another parameter and define
\[
\varTheta_\tau = \left\{ (\alpha,\beta) \in \cM_\cG \times\cM_\cH: |\{(x,y) \in \cG\times \cH~:~ux+vy=1\}| \ge \tau \right\}.
\]
 Hence,
\[
 \varTheta_\tau = \left\{ (\alpha,\beta) \in \cM_\cG \times\cM_\cH: \iota_{\cG,\cH}(\ell_{-a\beta^{-1},\beta^{-1}}) \ge \tau \right\}.
\]
By Lemma~\ref{lem:Mit} we have 
\begin{equation} \label{eq:The-tau}
|\varTheta_\tau|\tau \ll (GH)^{1/3}|\varTheta_\tau|^{2/3}
\end{equation}
provided $G^2H^2|\varTheta_\tau| \le p^3$ and $|\varTheta_\tau| \le GH$. Clearly the second condition is satisfied 
since $|\varTheta_\tau| \le |\cM_\cG||\cM_\cH|=(p-1)^2/(GH) \le GH$. 

We now suppose that  $G^2H^2|\varTheta_\tau| > p^3$. We define
\[
\cQ_\tau = \left\{ (\alpha,\beta) \in \F^*_p\times\F^*_p: \iota_{\cG,\cH}(\ell_{-a\beta^{-1},\beta^{-1}}) \ge \tau \right\}.
\]
We can then think of $\varTheta_\tau$ as a union of cosets. We have the number of incidences between $\cG \times \cH$ and lines $\ell_{-a\beta^{-1},\beta^{-1}}$ with $\alpha, \beta \in \cQ_\tau$ is at least
\[
|\cQ_\tau|\tau=GH|\varTheta_\tau|\tau >p^3G^{-1}H^{-1}\tau \ge p^3G^{-1}H^{-1}\Delta. 
\]
But from~\cite{StdeZe} we have the number of point--line incidences 
\[
|\cQ_\tau|\tau \ll |\cQ_\tau|^{1/2}GH+|\cQ_\tau|,
\]
hence
\[
p^3G^{-1}H^{-1}\Delta < |\cQ_\tau|\tau \ll G^2H^2\tau^{-1}<G^2H^2\Delta^{-1}.
\]
It follows that for this inequality to hold we need $\Delta^2\ll G^3H^3/p^3$. By choosing the constant $c$ in the definition of $\Delta$ small enough, we ensure this never happens. 

We now let $\tau_j=e^j\Delta$, for $j=0, 1, \dots, J$ and $J=\lceil \log(G^{1/2}H^{1/2}/\Delta) \rceil.$ We observe that the contributions from lines $\ell_{a,b}$, $a,b \in F^*_p$ is in one to one correspondence with those given by $\ell_{-a\beta^{-1},\beta^{-1}}$. 
Now from~\eqref{eq:The-tau} we have
\[
|\cQ_\tau|=GH|\varTheta_\tau|\ll G^2H^2\tau^{-3}.
\]
We also have $\tau_j \ge \tau_0 =\Delta \gg G^{3/2}H^{3/2}/p^{3/2}> GH/p$ for all $j$.
It follows that the contribution to $W$ is bounded by
\[
\sum_{j=0}^{J} |\cQ_{\tau_j}|\tau_{j+1}(\tau_{j+1} - GH/p)^2 \ll \sum_{j=0}^{J}|\cQ_{\tau_j}|\tau_{j+1}^3 \ll\sum_{j=0}^{J} G^2H^2\ll G^2H^2\log G.
\]
Substituting into~\eqref{eq:TGH} we have the required result.
\end{proof}

In particular, we see from Theorem~\ref{thm:subgroups}  that 
\[
\widetilde T(\cG,\cH)  \ll \left\{ 
\begin{array}{ll}
G^3H^3/p & \text{if $GH\ge p\log p$,} \\ 
G^2H^2\log{G} & \text{if $GH<p\log p$.} \end{array}\right. \,.
\] 

Since $\widetilde T(\cG,\cH) \le GH \widetilde D^\times(\cG,\cH)$, where $\widetilde D^\times(\cG,\cH)$ does not include the
zero solutions of $D^\times(\cG,\cH)$, we have the following. 

\begin{cor} \label{cor:Dgroup}
Let $\cG, \cH \subseteq \F^*_p$ be subgroups  of orders  $G\ge H$. Then 
\[
\widetilde D^\times(\cG,\cH) \ll \left\{ 
\begin{array}{ll}
G^4H^4/p & \text{if $GH\ge p\log p$,} \\ 
G^3H^3\log{G} & \text{if $GH<p\log p$.} \end{array}\right.
\]
\end{cor}

We mention the above result is only new for $G$ and $H$ falling either side of $(p \log p)^{1/2}$; see the proof of~\cite[Lemma~3.5]{MSS}.

\section{Applications}
\label{sec:applications}

Let us record what Theorem~\ref{thm:D_improved_new} gives for the exponential sums mentioned in the Introduction. 

\begin{theorem} 
	\label{thm:STsums}
For any sets $\cX, \cY, \cZ \subseteq \F_p$ as in~\eqref{eq:sets} 
and complex weights $\alpha$, $\beta$ and $\gamma$ as in~\eqref{eq:weights1} 
or $\rho$, $\sigma$ and $\tau$ as in~\eqref{eq:weights2}, for the sums 
 $S(\cX, \cY, \cZ; \alpha, \beta, \gamma)$  defined as in~\eqref{eq:SXYZ} or	
 $T(\cX, \cY, \cZ; \rho, \sigma, \tau)$  defined as in~\eqref{eq:TXYZ} we have	
\begin{itemize}
\item For $Y \leq p^{1/2}$
\[
S(\cX, \cY, \cZ;\alpha, \beta, \gamma)  \les p^{1/4} X^{3/4} Y^{21/26} Z^{1/2} E^{\times}(\cZ)^{1/8} \,. 
\]
\item For $Y , Z \leq p^{1/2}$
\[
T(\cX, \cY, \cZ ;  \rho, \sigma, \tau)  \les p^{1/8} X^{7/8} Y^{47/52} Z^{47/52} + XYZ^{3/4} \,.
\]
\end{itemize}
\end{theorem}

\begin{proof}
It is shown in the proof of~\cite[Theorem~1.1]{PetShp} that 
\[
S(\cX, \cY, \cZ;\alpha, \beta, \gamma) 
 \ll p^{1/4} X^{3/4} Z^{1/2} N(\cZ,\cY,\cY)^{1/4}. 
\]
Hence, by~\eqref{eq:N ED} we have
\[
S(\cX, \cY, \cZ;\alpha, \beta, \gamma)   \ll p^{1/4} X^{3/4} D^\times(\cY)^{1/8} Z^{1/2} E^{\times}(\cZ)^{1/8}\,.
\]
Furthermore, it  is shown in the proof of~\cite[Theorem~1.3]{PetShp} that 
\[
T(\cX, \cY, \cZ ;  \rho, \sigma, \tau)  \ll p^{1/8} X^{7/8} Y^{1/2} D^\times(\cY)^{1/16}  Z^{1/2} D^\times(\cZ)^{1/16} + XYZ^{3/4} \,.
\]
Using Theorem~\ref{thm:D_improved_new} proves both claims.
\end{proof}

The bound on $T(\cX, \cY, \cZ ;  \rho, \sigma, \tau)$ improves~\eqref{eq:Shkr} from~\cite{Shkr2}. Note that the range of non-triviality, that is the set of values of $X, Y, Z$ for which $|T(\cX, \cY, \cZ ;  \rho, \sigma, \tau)|$ is smaller than $XYZ$, obtained via the triangle-inequality, is $X^{13} Y^{10} Z^{10} \ges p^{13}$. Taking $\cX = \cY = \cZ$ gives that when $X \ges p^{13/33}$ the bound in Theorem~\ref{thm:STsums} is non-trivial. The example where $\cX = \cY = \cZ = \{1, \dots, X\}$ and all the weights equal 1, shows that $X \gg p^{1/3}$ is necessary.

The bound on $S(\cX, \cY, \cZ;\alpha, \beta, \gamma)$ improves~\cite{PetShp} under some conditions. For example, when $\cX=\cY=\cZ$,  the bound
\[
S(\cX, \cX, \cX;\alpha, \beta, \gamma) \ll p^{1/4} X^{19/8} \,,
\]
has been given in~\cite{PetShp}, while Theorem~\ref{thm:D_improved_new} gives
\[
S(\cX, \cX, \cX;\alpha, \beta, \gamma) \les p^{1/4} X^{107/52} E^{\times}(\cX)^{1/8} \,,
\] 
which is better when $E^{\times}(\cX) \leq X^{33/13 - \varepsilon}$ for some $\varepsilon >0$. Note that $33/13 > 5/2$ so this condition is not very restrictive as is known to be satisfied for many special sets. For example, see~\cite[Proposition~1]{Shkr1}, 
for shifted multiplicative subgroups of $\F_p^*$.

Another interesting example is given by the set 
$\cZ = \{z^{-1} + a:~ z \in \cI \} \subseteq \F_p$ of shifted reciprocals modulo $p$ of 
integers of an  interval $ \cI = [k+1, k+Z]$ (embedded in $\F_p$) with some integers $k$ and $Z \ge 1$.
Clearly, the equation 
\[
\(u^{-1} + a\) \(v^{-1} + a\)  = \(y^{-1} + a\) \(z^{-1} + a\), \qquad u, v,  y, z \in \cI\,,
\]
has $O(Z^2)$ solutions with 
\[
\(u^{-1} + a\) \(v^{-1} + a\)  = \(y^{-1} + a\) \(z^{-1} + a\)= 0 \mbox{ or } a^2 , \quad   u, v,  y, z \in \cI\,.
\]
Otherwise we 
reduce it to $O(Z^2)$ equations of the form 
\begin{equation}
 \label{eq:lambda}
\(y^{-1} + a\) \(z^{-1} + a\)  = \lambda , \qquad   y, z \in \cI\,,
\end{equation}
with some fixed $\lambda  \in \F_p \setminus \{ 0, a^2\}$. 
One verifies that~\eqref{eq:lambda} is equivalent to
\[
(y+b)(z+b) = \mu , \qquad   y, z\in \cI\setminus \{ 0\} \,,
\]
with 
\[
b = \frac{a}{a^2 -\lambda} \mand \mu = -   \frac{\lambda}{(a^2 -\lambda)^2} , 
\]  
which by  a result of   Cilleruelo and Garaev~\cite[Equation~(3)]{CillGar}, has
at most $Z^{3/2}p^{-1/2+o(1)} + Z^{o(1)}$ solutions. 
Hence 
\[
E^\times (\cZ) \le Z^{7/2}p^{-1/2+o(1)} + Z^{2+o(1)}. 
\]

\begin{theorem} 
	\label{thm:Tgroup}
For subgroups $\cF, \cG, \cH \subseteq \F_p$  with $G \geq H$ 
and complex weights $\rho$, $\sigma$ and $\tau$ as in~\eqref{eq:weights2}, for the	
 $T(\cF, \cG, \cH; \rho, \sigma, \tau)$  defined as in~\eqref{eq:TXYZ} we have	 
  \begin{itemize}  
\item For $GH \geq p\log p $, 
\[
T(\cF, \cG, \cH; \rho, \sigma, \tau)  \ll F^{7/8} GH + FGH^{3/4} \,. 
\]
\item For $GH < p \log p$,
\[
T(\cF, \cG, \cH ;  \rho, \sigma, \tau)  \les p^{1/8} F^{7/8} G^{7/8} H^{7/8} +FGH^{3/4} \,.
\]
\end{itemize}
\end{theorem} 

\begin{proof}
Again, we recall that it is shown in~\cite{PetShp} that 
\[
T(\cF, \cG, \cH ;  \rho, \sigma, \tau)  \ll p^{1/8} F^{7/8} G^{1/2} H^{1/2} \widetilde D^\times(\cG, \cH)^{1/8}+ FGH^{3/4} \,.
\]
Using Corollary~\ref{cor:Dgroup} we prove the result.
\end{proof}

The first bound of Theorem~\ref{thm:Tgroup} is non-trivial when $F \to \infty$ as $p \to \infty $, while the second bound is non-trivial when $FHG \ge p^{1+\varepsilon}$ for some fixed $\varepsilon >0$. The first part of the theorem improves~\cite[Lemma~3.5]{MSS} in the range $H < p^{1/2}$, and $GH \geq p \log p$ and leads to improved bounds on trinomial exponential sums in some ranges. The bound in~\cite[Lemma~3.5]{MSS} in the range $H < (p \log p)^{1/2} < G \leq F$ is
\[
T(\cF, \cG, \cH; \rho, \sigma, \tau)  \les p^{1/16} F^{7/8} G H^{7/8}
\]
and
\[
\frac{F^{7/8} GH}{p^{1/16} F^{7/8} G H^{7/8}} = \left( \frac{H^2}{p}\right)^{1/16} \,.
\]

As in~\cite[Theorem~1.5]{MSS},  we can use Theorem~\ref{thm:Tgroup} to  estimate exponential sums
 \begin{equation}
\label{eq:TrinomSum}
S_\chi(\Psi) = \sum_{x\in \F_p^*} \chi(x) \ep(\Psi(x)), 
 \end{equation} 
 with  trinomials 
 \begin{equation}
\label{eq:Trinom} 
\Psi(X) =aX^{k}+bX^\ell + cX^m, 
 \end{equation}
 where $\chi$ is an arbitrary multiplicative character of $\F_p^*$.

\begin{cor}
\label{cor:Bound3}   
Let $\Psi(X)$ be a trinomial of the form~\eqref{eq:Trinom} 
with $a,b,c  \in \F_p^*$.  
Define
\[d= \gcd(k,p-1), \qquad e = \gcd(\ell,p-1), \qquad  f =  \gcd(m,p-1)
\]
and
\[
g =\frac{d}{\gcd(d,f)},\qquad h =\frac{e}{\gcd(e,f)}.
\]
Suppose 
$f \ge g\ge h$, then  for sum~\eqref{eq:TrinomSum} we have
\[
S_\chi(\Psi) \ll  \left\{
\begin{array}{ll}
p^{7/8}f^{1/8}, & \text{if $gh \ge p  \log p$},\\
p(f/gh)^{1/8}\(\log p\)^{1/8}, & \text{if $gh < p  \log p$}. 
\end{array}
\right.
\]
\end{cor}
\begin{proof}
As in the proof of \cite[Theorem~1.5]{MSS}, adopting our new bound from Theorem \ref{thm:Tgroup}, for $gh > p\log p$ we have
\begin{align*}
S_\chi(\Psi) \ll p^{7/8}f^{1/8} + ph^{-1/4}.
\end{align*}
The first term clearly dominates when $h^2f > p$. Now,
\begin{align*}
h^2f \ge \frac{h f p \log p}{g} > hp \log p
\end{align*}
since $f \ge g$. Hence the first term dominates everywhere in this range.

For $gh < p \log p$ we have
\begin{align*}
S_\chi(\Psi) \ll p(f/gh)^{1/8}\(\log p\)^{1/8} + ph^{-1/4}.
\end{align*}
The first term clearly dominates when $f h \log^8 p > g$. Since $f\ge g$, the first term dominates everywhere in this range. This gives the required result.
\end{proof}

In particular, Corollary~\ref{cor:Bound3}    improves~\cite[Theorem~1.5]{MSS} for 
$ g \ge \(p  \log p\)^{1/2}>h$

Theorem~\ref{thm:D_improved_new} also improves 
the existing lower bound on $|(\cA-\cA)(\cA-\cA)|$ when $A \leq p^{1/2}$.

\begin{theorem} 
	\label{thm:ProdDiff}
For any set $\cA \subseteq \F_p$ of size $A \leq p^{1/2}$ we have
\[
|(\cA-\cA)(\cA-\cA)| \ges A^{20/13}
\]
\end{theorem}

\begin{proof}
By an application of the Cauchy-Schwarz inequality we get
\[
|(\cA-\cA)(\cA-\cA) | \geq \frac{A^8}{D^\times(\cA)} \ges A^{20/13}.
\]
\end{proof}

This improves the exponent $68/45 - \varepsilon$ for all $\varepsilon >0$ and $A \leq p^{9/16}$ give in~\cite[Theorem~27]{MPR-NRS}.

\section{Acknowledgement}

During the preparation of this work, S.M. was supported by  the Australian Government Research Training Program Scholarship,   
G.P.   by the NSF Award 1723016 (he also gratefully acknowledges the support from the RTG in Algebraic Geometry, Algebra, and Number Theory at the University of Georgia),
 I.D.S. by the  of the Russian Government  Grant N~075-15-2019-1926 and I.E.S. by Australian Research Council  Grant DP170100786.

\end{document}